\documentclass[reqno, oneside, 12pt]{amsart}

\usepackage[letterpaper]{geometry}
\usepackage{amsmath}
\usepackage{amssymb}
\usepackage{enumerate}
\usepackage{enumitem}
\usepackage{verbatim}
\usepackage{mathrsfs}
\usepackage{mathtools}
\usepackage{stmaryrd}

\mathtoolsset{centercolon}

\usepackage{setspace}

\numberwithin{equation}{section}

\newtheorem{theorem}{Theorem}
\newtheorem{corollary}[theorem]{Corollary}
\newtheorem{lemma}[theorem]{Lemma}
\newtheorem*{lemma*}{Lemma}
\newtheorem{proposition}[theorem]{Proposition}

\theoremstyle{remark}
\newtheorem*{remark*}{Remark}

\renewcommand{\bar}{\overline}

\newcommand{\Z}{\mathbb{Z}}
\newcommand{\Q}{\mathbb{Q}}

\newcommand{\SL}{\textup{SL}}
\newcommand{\GL}{\textup{GL}}

\newcommand{\pbar}{\overline{p}}
\DeclareMathOperator{\spt}{spt}
\DeclareMathOperator{\m}{m2}
\DeclareMathOperator{\mspt}{M2spt}
\DeclareMathOperator{\sptbar}{\bar{spt1}}

\newcommand{\hol}{\mathrm{hol}}
\newcommand{\ep}{\varepsilon}

\renewcommand{\pmatrix}[4]{\left( \begin{smallmatrix} #1 & #2 \\ #3 & #4 \end{smallmatrix} \right)}
\newcommand{\pMatrix}[4]{\left( \begin{matrix} #1 & #2 \\ #3 & #4 \end{matrix} \right)}

\newcommand{\pfrac}[2]{\left(\frac{#1}{#2}\right)}

\newcommand{\tleg}[2] {\left(\tfrac{#1}{#2}\right)}

\renewcommand{\)}{\right)}

\dedicatory{In  memory of Basil Gordon}
\date{}
\author{Scott Ahlgren}
\address{Department of Mathematics\\
University of Illinois\\
Urbana, IL 61801} 
\email{sahlgren@illinois.edu}

\author{Nickolas Andersen}
\address{Department of Mathematics\\
University of Illinois\\
Urbana, IL 61801} 
\email{nandrsn4@illinois.edu}

\title{Euler-like recurrences for smallest parts functions}

\begin{document}
\begin{abstract}  We obtain recurrences for smallest parts functions which resemble Euler's
recurrence for the ordinary partition function.  The proofs involve the holomorphic projection of non-holomorphic modular
forms of weight $2$.
 \end{abstract}
 \thanks{The first  author was supported by a grant from the Simons Foundation (\#208525 to Scott Ahlgren).}
\maketitle

\section{Introduction}

Let $p(n)$ denote the unrestricted partition function. One of the fundamental results in partition theory is Euler's recurrence, which states that for $n>0$ we have
\begin{equation}
	\sum_{k}(-1)^k p \left(n-\tfrac{k(3k+1)}{2}\right)=0.
	\label{eulerrec}
\end{equation}
The {\it smallest parts} function $\spt(n)$, which counts the number of smallest parts in the partitions of $n$, was introduced by Andrews \cite{Andrews:2008}. 
This and other smallest parts functions have been studied widely in recent years from a number of perspectives (see, e.g.  \cite{Ahlgren:2011, Ahlgren:2012,  Andrews:2012, BLO:Rank, BLO:Automorphic, Folsom:2008a, Garvan:2012} and the many references therein). 
Many of the  beautiful properties of these functions originate from  the fact that the associated generating functions are components of  mock modular forms of weight $3/2$.
  
Here we use the technique of holomorphic projection (as described by Sturm \cite{Sturm:1980} and Gross-Zagier \cite{GZ:Heegner}) to derive analogues of \eqref{eulerrec} for smallest parts functions.
The basic principle (also used recently in \cite{Andrews:2013} and \cite{Imamoglu:2013}) is that for a non-holomorphic modular form $f=f^++f^-$ written as a sum of holomorphic and non-holomorphic parts, we have $\pi_{\rm hol}(f)=f^++\pi_{\rm hol}(f^-)$.
If one can identify the holomorphic modular form $\pi_{\rm hol}(f)$ and can compute $\pi_{\rm hol}(f^-)$ explicitly, then a formula for $f^+$ results.
The simplest such analogue involves $\spt(n)$.  The associated holomorphic projection has been described (without proof) by Zagier \cite[\S 6]{Zagier:2009}; for completeness we give a brief account here.
 
Let
\begin{equation*}
	\eta(z):=q^\frac1{24}\prod_{n=1}^\infty(1-q^n)
\end{equation*}
denote Dedekind's eta function and let $E_2(z)$ be the quasimodular weight $2$ Eisenstein series on $\SL_2(\Z)$.
Define
\begin{equation*}
	F(z):=\sum_{n=1}^\infty\spt(n)q^{n-\frac1{24}}-\frac1{12}\cdot\frac{E_2(z)}{\eta(z)}
	+\frac{\sqrt{3i}}{2 \pi}\int_{- \overline{z}}^{i \infty} \frac{\eta(\tau)}{\left(z+\tau \right)^{\frac{3}{2}}}\, d\tau.
\end{equation*} 
Let $\varepsilon$ be the multiplier on $\SL_2(\Z)$ associated to the eta function.   
It can be shown (see \cite{Bringmann:2008} or \cite[\S 3]{Ahlgren:2011}) that $F(z)$ is  a weak harmonic Maass form of weight $3/2$ on $\SL_2(\Z)$ with multiplier $\overline\varepsilon$, so the function $\eta(z)F(z)$ transforms like a modular form of weight $2$ on $\SL_2(\Z)$. 
For positive integers $n$, define 
\begin{equation*} 
	a(n):=-\sum_{\substack{ab=6n \\ 0<a<b}}\tleg{12}{b^2-a^2}\cdot a.
\end{equation*}
We have
\[
	\sum_{n=1}^\infty a(n)q^n=q+2 q^2+q^3+2 q^4-q^5+3 q^6-2 q^7+2 q^8+q^9+q^{10}+\dots.
\]

Letting $E_2^*(z)$ denote the non-holomorphic Eisenstein series on $\SL_2(\Z)$, it can be shown that the holomorphic projection of  $\eta(z)F(z)+\frac1{12}E_2^*(z)$ is equal to $0$.
By computing this projection directly (using an argument similar to those given below) one can deduce that
\begin{equation*}
	\prod_{n=1}^\infty (1-q^n)\cdot \sum_{n=1}^\infty \spt(n)q^n=\sum_{n=1}^\infty a(n)q^n.
\end{equation*}

In other words, we have the following Euler-like recurrence  for $\spt(n)$, which is recorded in a slightly different form by Zagier \cite{Zagier:2009} and Andrews-Rhoades-Zwegers \cite[Thm. 11.1]{Andrews:2013}.
\begin{theorem}\label{eulersptthm} 
For $n>0$ we have
\[
	\sum_{k}(-1)^k \spt\(n-\tfrac{k(3k+1)}{2}\)=a(n).
\]
\end{theorem}

We will derive similar recurrences for other smallest parts functions. 
An {\it overpartition} is a partition in which the first occurrence of each part may be overlined.
Let $\overline p(n)$ denote the number of overpartitions of $n$ and let $\sptbar(n)$ denote the number of odd smallest parts in the overpartitions of $n$ (see \cite{BLO:Rank}).
Define a divisor function $s(n)$ by
\[
	s(n) := \sum_{d|n} \min\left(d,\frac nd\right),
\]
with the convention that $s(n)=0$ if $n\notin \Z$.
Define
\[
	b(n) := (-1)^{n+1}
	\begin{cases}
		2s(n) &\text{ if $n$ is odd},\\
		4s(n/4) &\text{ if } n\equiv 0 \pmod{4}, \\
		0 & \text{ if } n\equiv 2 \pmod{4}.
	\end{cases}
\]
Then we have the following analogue of \eqref{eulerrec} for $\sptbar(n)$.

\begin{theorem} \label{thm:spt1bar}
	For   $n>0$ we have
	\[
		\sum_k (-1)^k \sptbar(n-k^2) = b(n).
	\]
\end{theorem}
Theorem \ref{thm:spt1bar} is equivalent to the identity
\[
	\sum_{n\in \Z} (-1)^n q^{n^2}\,  \sum_{m=1}^\infty \sptbar(m) q^m = \sum_{n=0}^\infty b(n) q^n = 2q + 4q^3 - 4q^4 + 4q^5 + 4q^7 - 8q^8 + \ldots.
\]
Since we have 
\[
	\sum_{n=0}^\infty \pbar(n) q^n = \left( \sum_{n\in \Z} (-1)^n q^{n^2} \right)^{-1} = 
	1+2 q+4 q^2+8 q^3+14 q^4+24 q^5+40 q^6+\ldots,
\]
we obtain the following

\begin{corollary} For all $N>0$ we have
\[
	\sptbar(N) = \sum_{n+m=N} \pbar(n) b(m).
\]
\end{corollary}

Following \cite{BLO:Automorphic}, let $\m(n)$ denote the number of partitions of $n$ without repeated odd parts, and define $\mspt(n)$ as  the restriction of $\spt(n)$ to these partitions whose smallest part is even. Define
\[
	c(n) := \sigma (n) - \sigma (n/2) - \frac 12 s(2n) + s(n/2),
\] 
where $\sigma(n)$ denotes the usual sum of divisors function.

\begin{theorem} \label{thm:m2spt}
	For $n>0$ we have
	\[
		\sum_{k\geq 0} (-1)^{k(k+1)/2} \mspt\left( n-\tfrac{k(k+1)}{2} \right) = (-1)^n c(n).
	\]
\end{theorem}
We will prove the theorem by establishing the identity
\[
	\sum_{n=0}^{\infty} q^{n(n+1)/2} \sum_{m=1}^\infty (-1)^m \mspt(m) q^m 
		= \sum_{n=1}^\infty c(n) q^n = q^2 + q^3 + 3q^4 + 3q^5 + 4q^6 + \ldots .
\]
Since 
\[
	\left( \sum_{n=0}^\infty q^{n(n+1)/2} \right)^{-1} = \sum_{n=0}^\infty (-1)^n \m(n)q^n = 1-q+q^2-2 q^3+3 q^4-4 q^5+5 q^6 +\ldots,
\]
we obtain the following
\begin{corollary}
For all $n>0$ we have
\[
	\mspt(N) = \sum_{n+m=N} (-1)^m \m(n) c(m).
\]
\end{corollary}

\section{Preliminaries}

Let $k\in \Z$. For matrices $\gamma=\pmatrix abcd \in \GL_2^+(\Q)$ and functions $f$ on the upper half plane we define
\[
	\left(f \big|_k \gamma \right)(z) := \det(\gamma)^{\frac k2} (cz+d)^{-k} f\left(\frac{az+b}{cz+d}\right).
\]
We say that $f$ has weight $k$ for $\Gamma_0(N)$ if $f \big|_k \gamma = f$ for all $\gamma \in \Gamma_0(N)$. 
Let $E_2$ denote the weight $2$ quasi-modular Eisenstein series
\[
	E_2(z) := 1 - 24 \sum_{n=1}^\infty \sigma(n) q^n.
\]
Then the functions
\[
	E_2^*(z) := E_2(z) - \frac{3}{\pi y} \quad \text{ and } \quad E(z) := 2E_2(2z) - E_2(z)
\]
have weight $2$ for $\SL_2(\Z)$ and $\Gamma_0(2)$, respectively. Letting $W_2 := \pmatrix 0{-1}20$ denote the Fricke involution, we have $E \big|_2 W_2 = -E$ and
\[
	\left( E_2^* \big|_2 W_2 \right)(z) = 2E_2^*(2z). 
\]

Define
\begin{equation} \label{eq:def-G}
	G(z) := \sum_{n\geq 1} \sptbar(n)q^n + \frac{1}{12} \frac{\eta(2z)}{\eta^2(z)} (E_2(z) - 4E_2(2z)) + \frac{1}{2\sqrt{2}\pi i} \int_{-\bar{z}}^{i\infty} \frac{\eta^2(\tau)/\eta(2\tau)}{(-i(\tau+z))^\frac{3}{2}} d\tau
\end{equation}
and
\begin{multline} \label{eq:def-H}
	H(z) := \sum_{n\geq 1} (-1)^n \mspt(n) q^{n-\frac18}   \\  + \frac{1}{24} \frac{\eta(z)}{\eta^2(2z)} (E_2(2z) - E_2(z)) + \frac{1}{2\pi i} \int_{-\bar{z}}^{i\infty} \frac{\eta^2(2\tau)/\eta(\tau)}{(-i(z+\tau))^\frac{3}{2}} d\tau.
\end{multline}
By work of Bringmann, Lovejoy, and Osburn \cite{BLO:Automorphic}, these functions are harmonic weak Maass forms of weight $3/2$ (see, for example, \cite{Ono:2009} for details). In the notation of \cite{BLO:Automorphic}, $G(z) = -\frac14 \bar{\mathcal{M}}(z)$ and (correcting a sign error) $H(z) = \bar{\mathcal{M}}_2(z/8)$. From the proof of Lemma~6.1 of \cite{BLO:Automorphic}, we have
\begin{equation} \label{eq:GH-fricke}
(-i\sqrt{2}z)^{-\frac32}\, G(-1/2z)	  = -2^{\frac14} H(z).
\end{equation}

We use this fact to obtain the following proposition.
\begin{proposition}
The functions
\[
	g(z) := \frac{\eta^2(z)}{\eta(2z)} G(z) \quad \text{ and } \quad h(z) := \frac{\eta^2(2z)}{\eta(z)} H(z)
\]
have weight $2$ for $\Gamma_0(2)$.
\end{proposition}
\begin{proof}
The group $\Gamma_0(2)/\{\pm I\}$ is generated by the matrices $\pmatrix 1101$ and $\pmatrix 1021$.
By \eqref{eq:def-G} and \eqref{eq:def-H} we have $g(z+1)=g(z)$ and $h(z+1)=h(z)$.
To check the transformation under $\pmatrix 1021$, we write
\begin{equation*}
		\pMatrix 1021 = W_2 \pMatrix 1{-1}01 W_2^{-1}.
\end{equation*}
Using \eqref{eq:GH-fricke} and the fact that $\eta(-1/z)=\sqrt{-iz} \, \eta(z)$, we find that
\begin{equation}\label{eq:gfricke}
	g(z)\big|_2W_2= 2 h(z),
\end{equation}
from which 
\[
	g(z)\big|_2\pMatrix 1021=g(z).
\]
The same is true for $h(z)$, and the proposition follows. 
\end{proof}

We introduce the holomorphic projection operator.
Let $k\geq 2$ be an even integer. Suppose that $\phi(z)$ has weight $k$ for $\Gamma_0(N)$ and has Fourier expansion
\[
	\phi(z) = \sum_{m\in \Z} \alpha(m,y) q^m.
\]
Define
\[
	\pi_\hol(\phi) := \sum_{m=1}^\infty a(m) q^m,
\]
where
\begin{equation} \label{eq:hol-proj}
	a(m) := \frac{(4\pi m)^{k-1}}{(k-2)!} \int_0^\infty \alpha(m,y) e^{-4\pi m y} y^{k-2} \, dy
\end{equation}
(provided that this integral converges).
The next lemma is  Proposition 5.1 of   \cite{GZ:Heegner} if $k>2$.  When $k=2$ it follows from the proof of Proposition  6.2, loc. cit. (note that condition \eqref{eq:coeffbound} ensures that the limit and integral at the bottom of page 296 may be interchanged).

\begin{lemma} \label{lem:gz}
Suppose that $k\geq 2$. Suppose that $\phi(z)$ has weight $k$ for $\Gamma_0(N)$ and satisfies
\[
	(\phi \big|_k \gamma)(z) \ll y^{-\ep} \  \text{ as } y\to \infty
\]
for some $\ep>0$ and for all $\gamma \in \SL_2(\Z)$. 
If $k=2$, suppose in addition that for some $\varepsilon'>0$ we have	
\begin{equation}\label{eq:coeffbound}
	\alpha(m,y) \ll_m y^{-1+\ep'}\quad \text{as} \quad y\to 0 \quad\text{for all} \quad m>0.
\end{equation}
Then $\pi_\hol(\phi)$ is a weight $k$ cusp form on $\Gamma_0(N)$.
\end{lemma}

\section{Proof of Theorem \ref{thm:spt1bar}}
Write $G=G^+ + G^-$, where
\[
	G^-(z)= \frac{1}{2\sqrt{2}\pi i} \int_{-\bar{z}}^{i\infty} \frac{\eta^2(\tau)/\eta(2\tau)}{(-i(\tau+z))^\frac{3}{2}} d\tau
\]
is the non-holomorphic part. Since $\eta^2(z)/\eta(2z)=1+2\sum_{n=1}^\infty (-1)^nq^{n^2}$, a computation gives
\[
	G^-(z) = \frac{1}{2\pi\sqrt{y}} + \frac{1}{\sqrt{\pi}} \sum_{n=1}^\infty (-1)^n n \beta(n^2 y) q^{-n^2},
\]
where $\beta(y) := \Gamma(-1/2,4\pi y)$ is the incomplete gamma function. Then
\[
	g(z) = \frac{\eta^2(z)}{\eta(2z)} G(z) = \frac{\eta^2(z)}{\eta(2z)} \sum_{n=1}^\infty \sptbar(n) q^n + \frac{1}{12}\left( E_2(z) - 4E_2(2z) \right) + \sum_{N\in \Z} B(N,y) q^N,
\]
where
\begin{equation} \label{eq:B-cases}
	B(N,y) = \frac{(-1)^N}{\sqrt{\pi}}
	\begin{dcases}
		2 \sum_{\substack{n^2-m^2=N \\ m>n\geq 1}} m \beta(m^2 y) + \delta_\square(|N|) \sqrt{|N|}\beta(|N|y)  &\text{ if }N<0,\\
		\frac{1}{2\sqrt{\pi y}} + 2 \sum_{m\geq 1} m\beta(m^2 y) &\text{ if }N=0,\\
		2 \sum_{\substack{n^2-m^2=N \\ n>m\geq 1}} m \beta(m^2 y) + \delta_\square(N) \frac{1}{\sqrt{\pi y}} &\text{ if }N>0.
	\end{dcases}
\end{equation}
Here $\delta_\square(N)=1$ if $N$ is a square, and $0$ otherwise. 
Since  $\beta(y) \sim (4\pi y)^{-3/2} e^{-4\pi y}$ as $y\to \infty$, we have
\[
	\sum_{\substack{n^2-m^2=N \\ n,m\geq 1}} m\beta(m^2 y) 
		 \ll y^{-\frac32}\sum_{\substack{n^2-m^2=N \\ n,m\geq 1}} \frac{1}{m^2}  e^{-4\pi m^2 y},
\]
where the implied constants here and in the rest of the paragraph are absolute.
Since $n^2-(n-1)^2=2n-1$, the equation $n^2-m^2=N$ implies that $n,m\leq (|N|+1)/2$. 
If $N>0$, then this sum is $\ll Ny^{-3/2}$. If $N<0$ then we have $m^2>-N$ for each term in the sum, from which it follows that the sum is $\ll |N|y^{-3/2} e^{4\pi Ny}$.
We conclude that as $y\to \infty$, we have
\begin{equation} \label{eq:bound-B}
	B(N,y) \ll 
	\begin{cases}
		|N|\,y^{-\frac32} \,  e^{4\pi N y} &\text{ if }N<0,\\
		y^{-\frac12}+Ny^{-\frac32} &\text{ if } N\geq 0.
	\end{cases}
\end{equation}

Define
\begin{align*}
	\hat{g}(z) :&= g(z) + \frac 16 E(z) + \frac{1}{12} E_2^*(z) \\
		&= \frac{\eta^2(z)}{\eta(2z)} \sum_{n=1}^\infty \sptbar(n) q^n - \frac{1}{4\pi y} + \sum_{N\in \Z} B(N,y) q^N.
\end{align*}
By \eqref{eq:bound-B} we have $\hat{g}(z) \ll y^{-1/2}$ as $y\to \infty$. From \eqref{eq:gfricke} we obtain
\[
	\hat{g} \big|_2 W_2 
		= 2 h(z) + \frac{1}{6} \left(E_{2}(z) - E_2(2z)\right) - \frac{1}{4\pi y}.
\]
Therefore $\hat g \big|_{2} W_{2} \ll y^{-1}$ as $y\to\infty$ since $h(z)$ decays exponentially at $\infty$.

For $N>0$, we have the bound 
\[
	B(N,y) \ll_N y^{-\frac12} \quad \text{ as } y\to 0
\]
since $\lim\limits_{y\to 0}\beta(y) = -2\sqrt{\pi}$. 
Therefore we may apply Lemma \ref{lem:gz} to obtain
\[
	\pi_\hol(\hat{g}) = 0
\]
since there are no nontrivial cusp forms of weight $2$ on $\Gamma_0(2)$. 

We may also compute $\pi_\hol(\hat{g})$ using \eqref{eq:hol-proj}.
Since $\pi_\hol$ leaves holomorphic functions unchanged, we have
\[
	\pi_\hol(\hat{g}) = \frac{\eta^2(z)}{\eta(2z)} \sum_{n=1}^\infty \sptbar(n) q^n + \pi_\hol\left( - \frac{1}{4\pi y} + \sum_{N\in \Z} B(N,y) q^N \right).
\]
By \eqref{eq:hol-proj} we have
\[
	\pi_\hol\left( - \frac{1}{4\pi y} + \sum_{N\in \Z} B(N,y) q^N \right) = \sum_{N=1}^\infty \left(4\pi N \int_0^\infty B(N,y) e^{-4\pi N y} \, dy\right) q^N.
\]
By \eqref{eq:B-cases}, the coefficient of $q^N$ above is
\begin{equation} \label{eq:B-integral}
	(-1)^N 8\sqrt{\pi} N \sum_{\substack{n^2-m^2=N \\ n,m\geq 1}} m \int_0^\infty \beta(m^2 y) e^{-4\pi N y} \, dy + \delta_\square(N) (-1)^N 4N \int_0^\infty y^{-\frac12} e^{-4\pi N y} \, dy.
\end{equation}
The second integral evaluates to $\frac1{2\sqrt{N}}$ and the first is evaluated using the following lemma. The proof is routine (some care  is required to justify the change in the order of integration).

\begin{lemma}\label{lem:gamma-integral}
	If $A,B>0$ then
	\begin{equation} \label{eq:gamma-int}
		\int_0^\infty \beta(Ay) e^{-4\pi B y} \, dy = \frac{1}{2\sqrt{\pi}B} \left( \sqrt{1+\tfrac BA} - 1 \right).
	\end{equation}
\end{lemma}

Therefore \eqref{eq:B-integral} becomes
\begin{gather*}
	(-1)^N 4 \sum_{\substack{n^2-m^2=N \\ n,m\geq 1}} m \left( \sqrt{1+\tfrac{N}{m^2}} - 1 \right) + \delta_\square(N) (-1)^N 2\sqrt{N} \\ 
	= (-1)^N 2\left( 2\sum_{\substack{n^2-m^2=N \\ n,m\geq 1}} (n-m) + \delta_\square(N)\sqrt{N}\right).
\end{gather*}
It remains to show that this evaluates to $-b(N)$. If $N\equiv 2\pmod 4$, then the sum is empty and $\delta_\square(N)=0$. 
If $N$ is odd, then $n-m$ runs over all divisors of $N$ which are less than $\sqrt{N}$. In this case we have
\[
	2\sum_{\substack{n^2-m^2=N \\ n,m\geq 1}} (n-m) + \delta_\square(N)\sqrt{N} = \sum_{d|N} \min\left(d,\frac Nd\right).
\]
Finally, if $4|N$ then each  $n-m$ is even. 
Letting $r=\frac{n-m}{2}$ and $s=\frac{n+m}{2}$, we find that
\begin{equation*}
	\sum_{\substack{n^2-m^2=N \\ n,m\geq 1}} (n-m) 
		= \sum_{\substack{rs=N/4 \\ 0<r<s}} 2r 
		= \sum_{d|\frac N4} \min\left(d,\tfrac {N/4}d\right) - \delta_\square(N) \sqrt{\tfrac N4}.
\end{equation*}
\qed

\section{Proof of Theorem \ref{thm:m2spt}}

We proceed as in the proof of Theorem \ref{thm:spt1bar}. Write $H=H^+ + H^-$, where
\begin{align*}
	H^-(z) 
	= \frac{1}{2\pi i} \int_{-\bar{z}}^{i\infty} \frac{\eta^2(2\tau)/\eta(\tau)}{(-i(z+\tau))^{\frac32}}\,  d\tau.
\end{align*}
Since $\eta^2(2z)/\eta(z) = \displaystyle\sum_{\text{odd } n\geq 1} q^{n^2/8}$, we have
\begin{align*}
	H^-(z) = \frac{1}{4\sqrt{\pi}} \sum_{\text{odd }n\geq 1} n \beta\pfrac{n^2 y}{8} q^{-\frac{n^2}{8}}.
\end{align*}
Define $\hat{h}(z) := h(z) - \frac{1}{24}(E(z) - E_2^*(z))$.
Then \eqref{eq:def-H} gives
\[
	\hat{h}(z) = \frac{\eta^2(2z)}{\eta(z)}\sum_{n=1}^\infty (-1)^n \mspt(n) q^{n-\frac18} + \frac{1}{24}(E_2(z) - E_2(2z)) - \frac{1}{8\pi y} + \sum_N C(N,y) q^N,
\]
where
\[
	C(N,y) = \frac1{4\sqrt{\pi}}\sum_{\substack{n^2-m^2=8N \\ n,m\geq 1 \text{ odd}}} m \beta \pfrac{m^2 y}{8}.
\]
By an argument similar to that which gives \eqref{eq:bound-B}, we find that as $y\to \infty$ we have 
\[
	C(N,y) \ll 
	\begin{cases}
		|N|\, y^{-\frac32}e^{4\pi N y} & \text{ if  }\ N<0, \\
		y^{-\frac32} & \text{ if  }\ N= 0,\\
		N y^{-\frac32} & \text{ if  }\ N> 0.\\
	\end{cases}
\]
Thus we have $\hat{h}(z) \ll y^{-1}$ as $y\to \infty$. We have
\[
	\hat{h} \big|_2 W_2  = \frac 12 g + \frac{1}{24}(E(z) + 2E_2(2z)) - \frac{1}{8\pi y}.
\]
Therefore $\hat{h} \big|_2 W_2 \ll y^{-1/2}$ as $y\to \infty$ since the constant term of $g(z)$ is $-1/4$ and the constant term of $E(z)+2E_2(2z)$ is $3$. 
For $N>0$ we have the bound $C(N,y) \ll_N 1$ as $y\to 0$. 
Therefore, we may apply Lemma \ref{lem:gz} to conclude that $\pi_\hol(\hat{h})=0$.

Using \eqref{eq:hol-proj}, we find that
\begin{equation} \label{eq:pi-hol-h}
	0 = \pi_\hol(\hat{h}) = \frac{\eta^2(2z)}{\eta(z)}\sum_{n=1}^\infty (-1)^n \mspt(n) q^{n-\frac18} + \frac{1}{24}(E_2(z) - E_2(2z)) + \sum_{N=1}^\infty C(N) q^N,
\end{equation}
where
\[
	C(N) = \sqrt{\pi}N \int_0^\infty \sum_{\substack{n^2-m^2=8N \\ n,m\geq 1 \text{ odd}}} m \beta\pfrac{m^2 y}{8} e^{-4\pi N y} \, dy.
\]
By Lemma \ref{lem:gamma-integral} we obtain
\[
	C(N) = \frac 12 \sum_{\substack{n^2-m^2=8N \\ n,m\geq 1 \text{ odd}}} (n-m).
\]
Writing $u=\frac{n-m}{2}$ and $v=\frac{n+m}{2}$ gives
\[
	C(N) = \sum_{\substack{uv=2N \\ u<v \\ u+v \text{ odd}}} u = \frac 12 s(2N) - s(N/2).
\]
From \eqref{eq:pi-hol-h} we conclude that
\[
	\frac{\eta^2(2z)}{\eta(z)}\sum_{n=1}^\infty (-1)^n \mspt(n) q^{n-\frac18} = \sum_{n=1}^\infty c(n) q^n.
\]
\qed

\bibliographystyle{plain}
\bibliography{spt_recurrence_bib.bib}

\end{document}